\documentclass[10pt]{amsart}

\usepackage[utf8]{inputenc} 
\usepackage{amssymb}        
\usepackage{amsthm}         
\usepackage{mathtools}      
\usepackage{physics}        
\usepackage{asymptote}      
\usepackage{subcaption}     
\usepackage{xcolor}         
\usepackage[
    hyperfootnotes=false      
]{hyperref}                 
\usepackage{cleveref}       
\usepackage{microtype}      
\usepackage{booktabs}       
\usepackage[all]{xy}        
\usepackage{enumitem}       
\usepackage{thm-restate}    

\usepackage[
    maxbibnames = 99, 
    maxcitenames = 2,
    backend = biber
]{biblatex}
\makeatletter
\renewbibmacro*{related:bytranslator}[1]{%
  \entrydata{#1}{%
    \renewbibmacro*{name:hook}[1]{%
      \ifnumequal{\value{listcount}}{1}
        {\begingroup
         \mkrelatedstring%
         \lbx@initnamehook{#1}%
         \endgroup}
        {}}%
    \printnames[bytranslator]{translator}%
    \clearname{translator}%
    \setunit*{\addspace\bibstring[\mkrelatedstring]{astitle}\space}%
    \usebibmacro{title}%
    \ifentrytype{article}
      {\usebibmacro{in:}%
       \usebibmacro{journal+issuetitle}}
      {}
    \ifentrytype{inbook}
      {\usebibmacro{in:}%
       \usebibmacro{bybookauthor}%
       \newunit\newblock
       \usebibmacro{maintitle+booktitle}%
       \newunit\newblock
       \usebibmacro{byeditor+others}%
       \newunit\newblock
       \printfield{edition}}
      {}
    \ifentrytype{incollection}
      {\usebibmacro{in:}%
       \usebibmacro{maintitle+booktitle}%
       \newunit\newblock
       \usebibmacro{byeditor+others}%
       \newunit\newblock
       \printfield{edition}%
       \newunit
       \iffieldundef{maintitle}
         {\printfield{volume}%
          \printfield{part}}
         {}}
      {}
    \ifentrytype{inproceedings}
      {\usebibmacro{in:}%
       \usebibmacro{maintitle+booktitle}%
       \newunit\newblock
       \usebibmacro{event+venue+date}%
       \newunit\newblock
       \usebibmacro{byeditor+others}%
       \newunit\newblock
       \iffieldundef{maintitle}
         {\printfield{volume}%
          \printfield{part}}
         {}}
      {}
    \setunit{\addspace}%
    \printtext[parens]{%
      \printlist{location}%
      \iflistundef{publisher}
        {\setunit*{\addcomma\space}}
        {\setunit*{\addcolon\space}}%
      \printlist{publisher}%
      \setunit*{\addcomma\space}%
      \printdate}
      \newunit\newblock
      \printfield{note}%
      \newunit\newblock
      \printfield{chapter}%
      \setunit{\bibpagespunct}%
      \printfield{pages}
       \newunit\newblock  
       \printfield{url}%
      }}
\makeatother
\usepackage{filecontents}

\DeclareFieldFormat[article]{volume}{\mkbibbold{#1}}
\DeclareFieldFormat[article]{pages}{#1}
\DeclareFieldFormat[incollection]{pages}{#1}

\makeatletter
\patchcmd{\@setref}{\bfseries ??}{\bfseries\color{red} ?????}{}{}
\makeatother
\makeatletter
\def\l@subsection{\@tocline{2}{0pt}{2.5pc}{5pc}{}}
\renewcommand\tocchapter[3]{%
  \indentlabel{\@ifnotempty{#2}{\ignorespaces#2.\quad}}#3%
}
\newcommand\@dotsep{4.5}
\def\@tocline#1#2#3#4#5#6#7{\relax
  \ifnum #1>\c@tocdepth 
  \else
    \par \addpenalty\@secpenalty\addvspace{#2}%
    \begingroup \hyphenpenalty\@M
    \@ifempty{#4}{%
      \@tempdima\csname r@tocindent\number#1\endcsname\relax
    }{%
      \@tempdima#4\relax
    }%
    \parindent\z@ \leftskip#3\relax \advance\leftskip\@tempdima\relax
    \rightskip\@pnumwidth plus1em \parfillskip-\@pnumwidth
    #5\leavevmode\hskip-\@tempdima{#6}\nobreak
    \leaders\hbox{}\hfill
    \nobreak
    \hbox to\@pnumwidth{\@tocpagenum{#7}}\par
    \nobreak
    \endgroup
  \fi}
\makeatother
\AtBeginDocument{%
\makeatletter
\expandafter\renewcommand\csname r@tocindent0\endcsname{0pt}
\makeatother
}
\def\l@subsection{\@tocline{2}{0pt}{2.5pc}{5pc}{}}

\definecolor{linkblue}{HTML}{00356B}
\definecolor{linkgold}{HTML}{DA9100}
\definecolor{linkred}{RGB}{159,  29, 53}
\hypersetup{
	colorlinks = true,
	linkcolor = linkred,
	citecolor = linkgold,
	urlcolor = linkblue
}

\theoremstyle{plain}
\newtheorem{theorem}{Theorem}[section]
\crefname{theorem}{Theorem}{Theorems}
\newtheorem{conjecture}[theorem]{Conjecture}
\crefname{conjecture}{Conjecture}{Conjectures}
\newtheorem{proposition}[theorem]{Proposition} 
\crefname{proposition}{Proposition}{Propositions}
\newtheorem{corollary}[theorem]{Corollary} 
\crefname{corollary}{Corollary}{Corollaries}
\newtheorem{lemma}[theorem]{Lemma} 
\crefname{lemma}{Lemma}{Lemmas}

\theoremstyle{definition}
\newtheorem{example}[theorem]{Example}
\crefname{example}{Example}{Examples}
\newtheorem{definition}[theorem]{Definition}
\crefname{definition}{Definition}{Definitions}

\theoremstyle{remark}
\newtheorem*{remark}{Remark}

\crefname{appendix}{Appendix}{Appendices}
\crefname{section}{Section}{Sections}
\crefname{figure}{Figure}{Figures}
\crefname{equation}{Equation}{Equations}

\newcommand{\field}[1]{\mathbf{#1}}

\renewcommand{\H}{\field{H}}
\newcommand{\C}{\field{C}}

\newcommand{\poly}[1]{{#1}}
\newcommand{\hull}{H}
\newcommand{\ptops}[2]{%
  \texorpdfstring{\unexpanded{\unexpanded{#1}}}{#2}%
}

\def\titletext{Optimal Monohedral Tilings of Hyperbolic~Surfaces}
\title[Opt. Monohedral Tilings of Hyp. Surfaces]{\MakeUppercase{\titletext}}

\author[Di Giosia]{Leonardo {Di Giosia}}
\author[Habib]{Jahangir Habib}
\author[Hirsch]{Jack Hirsch}
\author[Kenigsberg]{Lea Kenigsberg}
\author[Li]{Kevin Li}
\author[Pittman]{Dylanger Pittman}
\author[Petty]{Jackson Petty}
\author[Xue]{Christopher Xue}
\author[Zhu]{Weitao Zhu}

\makeatletter
\let\@wraptoccontribs\wraptoccontribs
\makeatother

\address{Rice University, Department of Mathematics-MS 136, Box 1892, Houston, TX 77251-1892, USA}
\email[Leonardo {DiGiosia}]{lsd2@rice.edu}

\address{Department of Mathematics and Statistics, Williams College, Bascom House \\ 33 Stetson Court \\ Williamstown \\ MA 01267 \\ USA}
\email[Jahangir Habib]{jih1@williams.edu}

\address{Department of Mathematics \\ Yale University \\ New Haven, CT 06510, USA}
\email[Jack Hirsch]{jack.hirsch@yale.edu}
\email[Kevin Li]{k.li@yale.edu}
\email[Jackson Petty]{jackson.petty@yale.edu}
\email[Christopher Xue]{christopher.xue@yale.edu}

\address{Department of Mathematics, Columbia University, Room 509, MC 4406, 2990 Broadway, New York, NY 10027, USA}
\email[Lea Kenigsberg]{lea@math.columbia.edu}
\email[Weitao Zhu]{wz2453@math.columbia.edu}

\address{Department of Mathematics, Mail Stop: 1131-002-1AC, Emory University, Atlanta, GA 30322, USA}
\email[Dylanger Pittman]{dylanger.skyler.pittman@emory.edu}

\date{\today}

\begin{document}

\begin{abstract}
    The hexagon is the least-perimeter tile in the Euclidean plane for any
    given area. On hyperbolic surfaces, this ``isoperimetric'' problem differs 
    for every given area, as solutions do not scale. Cox conjectured that a 
    regular $k$-gonal tile with 120-degree angles is isoperimetric. For area 
    $\pi/3$, the regular heptagon has 120-degree angles and therefore tiles many hyperbolic surfaces. For other areas, we show the existence of many tiles but provide no conjectured optima.
    
    On closed hyperbolic surfaces, we verify via a reduction argument using cutting and pasting transformations and convex hulls that the regular 
    $7$-gon is the optimal $n$-gonal tile of area $\pi/3$ for $3\leq n \leq 10$.
    However, for $n>10$, it is difficult to rule out non-convex $n$-gons that tile irregularly.
   

\end{abstract}

\maketitle
\tableofcontents


\section{Introduction}
In \citeyear{hales} \textcite{hales} proved that the regular hexagon is the 
least-perimeter, unit-area tile of the plane, and further that
no such tiling of a flat torus can do better. Efforts to generalize this result to hyperbolic surfaces have to date been unsuccessful (see 
\cref{sect:monohedral}). We focus on monohedral tilings (by a single prototile)
and address the conjecture that a regular $k$-gon with $120^\circ$ angles is
optimal. Unfortunately, regular polygonal tiles of the hyperbolic plane $\H^2$ cover only a countable set of areas. 
We prove that \emph{equilateral} $2n$-gonal tiles ($n \geq 2$) cover large 
intervals of areas; for example, there are equilateral $12$-gonal tiles for all 
of the possible areas from $0$ to $10\pi$, except possibly the interval 
$(4\pi, 5\pi]$ (see \cref{sect:hyperbolic-plane}).

Our ideal regular polygons tile many closed hyperbolic surfaces, where we address the following conjecture:
\begin{conjecture}\label{con:main-result} 
Any non-equivalent tile of area $\pi/3$ of a closed hyperbolic surface has more perimeter than the regular heptagon $\poly{R}_7$. 

\end{conjecture}
Our \cref{prop:heptagon-best} proves that the regular $7$-gon with $120^\circ$ angles is optimal in comparison with all $n$-gons of area $\pi/3$ for $n \leq 10$. 
A subsequent paper [X] simultaneously proves \cref{con:main-result} in comparison with polygons of any number $n$ of sides and generalizes the result from $7$ to all $k \geq7$.

\subsection*{Methods}
To obtain equilateral $2n$-gonal tiles of $\textbf{H}^2$, it suffices by \textcite{MM} (\cref{prop:margulis}) to construct equilateral $2n$-gons with angles summing in various combinations to $2\pi$. \cref{prop:evenequilateral} actually shows there is an equilateral $2n$-gon with any repeated sequence of angles (so that opposite angles are equal) as long as the exterior angles sum to less than $2\pi$. The careful induction argument considers the effects as the constant sidelength $\ell$ approaches $0$ and $\infty$.


To prove $\poly{R}_7$ is the optimal tile of an appropriate closed hyperbolic surface, \cref{prop:reg-is-best} first verifies that among $n$-gons of given area, the regular one minimizes perimeter.
It follows easily that $\poly{R}_7$ has less perimeter than all other $n$-gonal tiles for $n \leq 7$.
For $n>7$, we show that in an $n$-gonal tiling there are on average at least $n-7$ vertices of 
degree $2$ per tile.
In particular for $n \geq 8$, an $n$-gonal tile has a concave angle. This means that the convex hull of an octogonal tile (see \cref{sec:octagonal-tiles}) has at most $7$ sides with generally more area and perimeter than $\poly{R}_7$.
Similarly, if a $9$-gonal tile (see \cref{sec:nonagonal-tiles}) has two or more concave angles, it has more perimeter than $Q_7$. If it instead has one concave angle, a flattening argument that fills in the concave angle and truncates the corresponding convex angle which fits into it  preserves area, reduces perimeter, and yields a heptagon which generally
has more perimeter than $\poly{R}_7$. 
Finally, for a $10$-gonal tiling (see \cref{sec:decagonal-tiles}), there may be many concave angles filled by many different convex angles, perhaps nested inside one another, complicating the flattening procedure. 
\cref{prop:heptagon-beats-decagon} reduces the analysis to six substantive cases, taking care to show that each may be flattened without resulting in self-intersecting shapes.

\textcite{hales} remarks that \citeauthor{toth43},  who proved the honeycomb conjecture for \emph{convex} cells  \cite{toth43}, predicted that general cells would involve considerable difficulties \cite[183]{tothFigures} and said that the conjecture had resisted all attempts at proving it \cite{tothBees}. Removing the convexity hypothesis is the major advance of Hales's work and of ours, although we consider just polygonal cells.



\begin{figure}[b]
    \centering
    \includegraphics[scale=0.5]{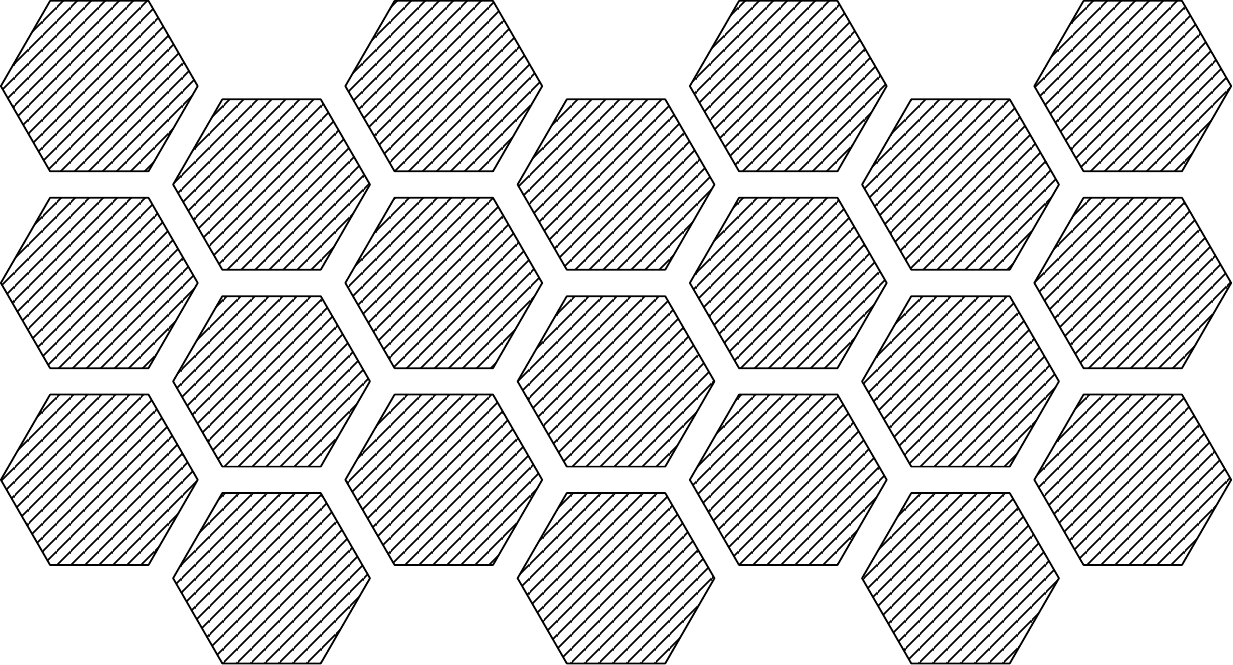}
    \caption{\citeauthor{hales} (\citeyear{hales}) proved that regular hexagons provide the least-perimeter equal-area tiling of the plane.}
    \label{fig:hales}
\end{figure}

\subsection*{Acknowledgements} This work is a product of the 2019 Summer Undergraduate Mathematics Research at Yale (SUMRY) and the 2016 Williams College
NSF SMALL Geometry Groups, under the guidance of Frank Morgan of Williams College. The authors greatly thank Prof.\ Morgan for his help and insight
over the many weeks we spent researching and writing this paper. We thank the National Science Foundation (NSF), the
Williams College Science Center, the John and Louise
Finnerty Fund, and Yale for support. We also thank the Mathematical Association of America, Stony Brook University, Williams College, the Young Mathematicians Conference, and Yale for supporting our trips to speak at MathFest 2016 and the 2019 Young Mathematicians Conference, both in Columbus, Ohio.

\section{Definitions}

\begin{definition}[Tiling]
Let $M$ be a closed Riemannian surface. A tiling of $M$ is an
embedded multigraph on $M$ with no vertices of degree 0 or 1. A tiling is \emph{polygonal}
if 
\begin{enumerate}
\item every edge is a geodesic;
\item every face is an open topological disk.
\end{enumerate}
  The oriented boundary of a face of a polygonal tiling is called a polygon.  A tiling is \emph{monohedral} if all faces are congruent. 
\end{definition}

\begin{remark}
All polygonal tilings are \emph{connected} multigraphs. When tiling a closed surface with a tile $\poly{Q}$, one copy $\poly{Q}^*$ might be edge-to-edge with itself. An example is tiling a hyperbolic two-holed torus with a single hyperbolic octagon, which has all eight of its vertices coinciding at one point, and each edge coinciding with another edge. A second example is tiling a one-holed torus by tiling the square fundamental region with thin vertical rectangles. The rectangle is edge-to-edge with itself at top and bottom, and the two vertices of a vertical edge coincide. This is consistent because a tiling is defined as a \emph{multi}graph. 
\end{remark}

It is often useful to consider $m$-gons which ``look like'' $n$-gons because
of angles of measure $\pi$, such as a rectangle which appears to be a
triangle because it has three angles of measure $2\pi/3$ and one angle of 
measure $\pi$. To clarify this situation, we introduce a notion of equivalence
to polygons.

\begin{definition}[Equivalent]
    Two polygons $\poly{Q}$ and $\poly{Q}'$ are \emph{equivalent} $\poly{Q}\sim\poly{Q}'$ \ if they are equal after the removal of all vertices of measure $\pi$.
\end{definition}

\begin{remark}
    We can't in general define away vertices of measure $\pi$; a vertex in a tiling could, for example, have angles $\pi, \pi/2, \pi/2$, so the vertex has to be there because of the $\pi/2$ angles.
\end{remark}

\begin{definition}[Convex Hull]
Let $R$ be a polygonal region in a closed hyperbolic surface $M$.
The convex hull $\hull (R)$ is taken in the hyperbolic plane
(with the minimal number of vertices). 
The convex hull of an $n$-gonal region $R$ is a $k$-gonal region for some~$k \leq n$. 
The convex hull has no less area and no more perimeter. 
\end{definition}

\begin{remark}[Existence] 
By standard compactness arguments, there is a perimeter-minimizing tiling for prescribed areas summing to the area of the surface, except that polygons may bump up against 
themselves and each other, possibly with angles of measure 0 and $2\pi$, in the limit.
We think that no such bumping occurs, but we have no proof. 
\end{remark}



\section{Hyperbolic Geometry}\label{sect:hyperbolic-geometry}

We begin with some basic results of hyperbolic geometry.
Of particular interest are formulae concerning the area and
perimeter of polygons in hyperbolic space. \cref{cor:combination-3-7}
proves that the regular heptagon is optimal {(\cref{con:main-result})} among polygons with seven or fewer sides.

\begin{proposition}
\label{prop:Gauss-Bonnet}
By the Gauss-Bonnet Theorem, an $n$-gon in the hyperbolic plane with 
interior angles $\theta_1, \dots, \theta_n$ has area
$(n-2)\pi-\sum  \theta_i$. In particular, a regular $n$-gon with interior angle $\theta$ has area 
\begin{equation} \label{eq:A-n-theta}
		A(n,\theta) = (n-2)\pi-n\theta. 
\end{equation}
\end{proposition}

\begin{proposition}[Law of Cosines]
\label{LoC}
If $l$ is the length of the side opposing angle $\theta_3$ in a triangle with interior angles $\theta_i$, then 
\[\cos{\theta_3} =\sin{\theta_1}\sin{\theta_2}\cosh{l}-\cos{\theta_1}\cos{\theta_2}.
\]
In particular, for right triangle $\triangle ABC$ with legs $a,b$, 
\[
\cosh{(a)}=\cos{(\angle A)}/\sin{(\angle B)}.
\]
\end{proposition} 

\begin{proposition}
\label{prop:perimeter-n-theta}
A regular $n$-gon with interior angle $\theta$ has perimeter
\begin{equation} \label{eq:P-n-theta}
P(n,\theta)=2n\cosh^{-1}\left(\frac{\cos(\pi/n)}{\sin(\theta/2)}\right).
\end{equation}
\end{proposition}

\begin{proof}
Connect the center of $\poly{Q}_n$ to each of its vertices to form $n$ isosceles triangles. 
Bisect the $n$ congruent triangles into $2n$ right triangles by connecting the center of the polygon to the bisector of each side of the polygon. 
Each triangle has interior angles $\pi/2, \pi/n$, and $\theta/2$. By \cref{LoC}, the length of the leg on the polygonal side of each of the $2n$ right triangles is $\cosh^{-1}(\cos(\pi/n)/\sin(\theta/2))$.
\end{proof}

\begin{definition} \label{def:AkPk}
For $k \geq 7$, let $A_k=A(k,2\pi/3)=(k-6)\pi/3$ and $P_k = P(k,2\pi/3)$ denote the area and perimeter of the regular $k$-gon $\poly{R}_k$ with angles $2\pi/3$.
\end{definition}

Regular $n$-gons are isoperimetric among $n$-gons.

\begin{proposition}[\cite{hirsch}, Prop.\ 3.7]\label{prop:reg-is-best}
    In the hyperbolic plane, the regular $n$-gon $\poly{Q}_n$ has less perimeter
    than any other $n$-gon $\poly{Q}$ of the same area.
\end{proposition}

\begin{corollary}
   \label{cor:combination-3-7}
   Tile a closed hyperbolic surface by polygons of area $\pi/3$ with 7 or fewer sides. Then each of those tiles has perimeter greater than or equal to that of the regular heptagon of area $\pi/3$.
\end{corollary}
\begin{proof}
    This corollary follows immediately from \cref{prop:reg-is-best}.
\end{proof}

Intuitively, the perimeter of regular $n$-gons of a fixed area should decrease
as~$n$ increases to approach the limiting bound of a circle, the most efficient
way to enclose a given area. Instead of performing computations of the perimeter
to prove this, we appeal to the fact that regular $n$-gons are more efficient 
than any other $n$-gons to show that this is, in fact, the case.

\begin{proposition} \label{prop:reg-decreasing}
The perimeter of a regular $n$-gon for a fixed area is decreasing as a 
function of $n$.
\end{proposition}

\begin{proof}
Let $\poly{Q}_n$ and $\poly{Q}_{n+1}$ be the regular polygons of a fixed
area with $n$ and $n+1$ sides. Let $\poly{Q}^*_{n+1}$ be an $(n+1)$-gon
formed by adding a vertex of measure $\pi$ to $\poly{Q}_n$.
By \cref{prop:reg-is-best},
\[  
    P(\poly{Q}_{n+1}) < P(\poly{Q}^*_{n+1}) = P(\poly{Q}_n). \qedhere 
\]
\end{proof}

\begin{remark}
    As expected, the perimeter of a regular $n$-gon of area $A$ is increasing as a function of $A$, for $0 < A < (n-2)\pi.$ By \cref{prop:Gauss-Bonnet} and \cref{prop:perimeter-n-theta}, the perimeter of the $n$-gon is
    \[2n\cosh^{-1}\left(\frac{\cos(\pi/n)}{\sin(((n-2)\pi-A)/2n)}\right),\] and it is increasing because $\cosh^{-1}$ and $\sin$ are increasing over $(0,\infty)$ and $(0, \pi/2)$ respectively. 
\end{remark}
\begin{corollary}
    \label{prop:3-k}
 	The regular $k$-gon has less perimeter than 
	any other $n$-gon of equal or greater area for $3 \le n \le k.$
\end{corollary}
\begin{proof}
    The corollary follows immediately from \cref{prop:reg-is-best} and \cref{prop:reg-decreasing}.
\end{proof}

The following corollary is an easy step toward \cref{con:main-result}.

\begin{proposition}

\label{prop:smolhull}
Consider an $n$-gon $\poly{Q}$ of area $A_k = (k-6)\pi/3$. If the convex hull $\hull(\poly{Q})$ has $k$ or fewer vertices, then $P(\poly{Q})\geq P_k = P(\poly{R}_k)$, with equality only if $\poly{Q}\sim\hull(\poly{Q})=\poly{R}_k$.
\end{proposition}

\begin{proof}
    Recall $\hull(\poly{Q})$ has no less area and at least as much perimeter as $\poly{Q}$. \cref{prop:3-k} finishes the proof. 
\end{proof}
\begin{corollary}
    \label{cor:n-7concave}
    If an $n$-gon $\poly{Q}$ of area $A_k = (k-6)\pi/3$ has at least $n-k$ concave angles, then $P(Q) \geq P_k$ with equality if and only if there are exactly $n-k$ such angles and they are all exactly $\pi,$ and hence $\poly{Q}\sim R_k$.
\end{corollary}

\begin{proof}
    The corollary is immediate from \cref{prop:smolhull}, because if $Q$ has at least $n-k$ concave angles, then the convex hull $H(Q)$ has $k$ or fewer vertices, with equality as claimed.
\end{proof}

\section{Monohedral Tilings of the Hyperbolic Plane} \label{sect:hyperbolic-plane}


We seek a least-perimeter tile of $\H^2$ of given area. For area $(n-6)\pi/3$ we conjecture that the regular $n$-gon, with $120^\circ$ angles, is best. For other areas there is no natural candidate. After \textcite{GS} and \textcite{MM} we prove the existence of
equilateral even-gonal tiles for wide ranges of areas.

\begin{conjecture}
\label{con:conjH2}
 
In\/ $\H^2$, the regular $n$-gon with $120^\circ$ angles has less perimeter than any non-equivalent tile of equal area. 
\end{conjecture}

The following proposition provides a necessary and sufficient condition for a regular polygon to tile the hyperbolic plane.
\begin{proposition}
	\label{prop:regularpoly}
	A regular polygon of interior angle $\theta$ tiles $\H^2$ 
	if and only if $\theta$ divides $2\pi$.
\end{proposition}

\begin{proof} Of course if $\poly{Q}$ tiles, 
$\theta$ divides $2\pi$. Conversely, as long as $\theta$ divides $2\pi$, you can form a tiling by surrounding one copy of $\poly{Q}$ with layers of additional copies.
Alternatively, this proposition follows directly from \cref{prop:margulis}. 
\end{proof}
 
\begin{remark}
    Similarly, if each angle of a triangle divides $\pi$, then the triangle tiles the hyperbolic plane.
\end{remark}

\begin{corollary}
An isosceles triangle $\poly{T}$ with angle $\theta_1$ dividing $2\pi$ and angles $\theta_2=\theta_3$ dividing $\pi$ tiles $\H^2$.
\end{corollary}

\begin{proof}
For such a $\poly{T}$, form a regular polygon $\poly{Q}$ with interior angle $2\theta_2$ by attaching $2\pi/\theta_1$ copies of $\poly{T}$ at the vertex of measure $\theta_1$. 
By Proposition \ref{prop:regularpoly}, $\poly{Q}$ tiles $\H^2$. Thus $\poly{T}$ tiles $\H^2$.
\end{proof}

\begin{remark}
The preceding propositions suggest several immediate but important observations.
\begin{enumerate}
    \item The areas of regular polygonal tiles are discrete except at the integer multiples of $\pi$. This follows from the fact that for bounded area, $n$ is bounded above for a regular $n$-gonal tile of that area, and the areas of regular $n$-gons approach $(n-2)\pi$.
    \item There are only finitely many regular polygonal tiles of given area. 
    \item\label{item:lincom} If a polygon tiles the hyperbolic plane, then each angle is included in some positive integer linear combination that equals $2\pi$.
    \item The converse of \eqref{item:lincom} is false. For instance, if a triangle $\poly{T}$ has angles $\theta_1,\theta_2, \theta_3$ satisfying a unique equation $\theta_1 + 3\theta_2 + 5\theta_3 = 2\pi$, $\poly{T}$ does not tile $\H^2$. This remark is a corollary of the following theorem of \citeauthor{GS}. 
\end{enumerate}
\end{remark}

\begin{theorem}[\textcite{GS}, Thm.\ 6.2] \label{thm:gs}
Suppose a hyperbolic triangle $\poly{T}$ with vertex angles $\alpha_i$ satisfies exactly one 
equation of the form $\sum k_i\theta_i = 2\pi$ with nonnegative integral coefficients. Then $\poly{T}$ tiles $\H^2$ if and only if all the coefficients are at least $2$ and congruent to one another modulo $2$.
\end{theorem}

We now relax the ``exactly one'' hypothesis of \cref{thm:gs}.

\begin{lemma} \label{lem:existence}
For any triangle $\poly{T}$ with angles $\alpha_1, \alpha_2,\alpha_3$ satisfying $\sum k_i\alpha_i = 2\pi$ for nonnegative integers $k_i$, there exists a scalene triangle $\poly{T}'$ whose angles satisfy this equation and no other nonnegative linear combination that sums to $2\pi$.
\end{lemma}

\begin{proof}
The constraint $\sum k_i\theta_i = 2\pi$ determines a plane $\Pi$ which intersects the region of possible hyperbolic triangle angles 
\[
    B = \left\{ \sum_{0<\theta_i} \theta_i <\pi \right\}
\]
of the first octant of $\theta_1\theta_2\theta_3$ space. For integers $(k_1',k_2',k_3')\neq (k_1,k_2,k_3)$, the collection of affine subspaces 
\[
\Pi \cap \left\{\sum k_i'\theta_i = 2\pi\right\}
\] 
is a countable set of lines and empty sets in $\Pi$. Choose $(\alpha_1',\alpha_2',\alpha_3') \in \Pi \cap B$ lying on no such line. The triangle $T'$ with angles $\alpha_i'$ is scalene because 
\[
(k_1-1)\alpha_1' + (k_2+1)\alpha_2' + k_3\alpha_3' \neq k_1\alpha_1'+k_2\alpha_2'+k_3\alpha_3'
\] 
implies $\alpha_1' \neq \alpha_2'$. A similar argument shows each $\alpha_i'$ is distinct.
\end{proof}
\begin{remark}
    Denote $\poly{T}'/m$ as the triangle with angles $1/m$ times those of $\poly{T}'.$ The statement in \cref{lem:existence} can be strengthened so that $\poly{T}'$ satisfies the given equation and no other \emph{rational} combination of its angles sums to $2\pi$. Then, by \cref{thm:gs}, if $k_i \ge 1$, $\poly{T}'/m$ tiles for all even $m$. If the coefficients are at least 2 and congruent modulo 2, then $\poly{T}'/m$ tiles for all positive integers $m$.
\end{remark}
\begin{proposition}[cf.\ Thm.\ 4.5 of \cite{GS}] \label{prop:unique-tiling}
Consider a triangle $\poly{T}$ and a tile $\poly{T}'$. Suppose that every 
nonnegative integral linear combination $\sum k_i\theta_i=2\pi$ satisfied by the
angles of $\poly{T}'$ is also satisfied by the angles of $\poly{T}$. Then 
$\poly{T}$ tiles in the same way.
\end{proposition}

\begin{proof}
First consider the case where $\poly{T}'$ is scalene. Then the triangle $\poly{T}$ tiles in exactly the same way as $\poly{T}'$. The angles still sum to $2\pi$ around every vertex, and the edges match because a tiling by the scalene triangle $\poly{T}'$ always matches an edge to itself. 

Now suppose $\poly{T}'$ is isosceles with angles  $\alpha_1, \alpha_2 = \alpha_3$. Since $T'$ tiles, some linear combination $\sum k_i \alpha_i = 2\pi$ with $k_2 \neq 0$. If $k_2=k_3,$ decrease $k_2$ and increase $k_3$ by 1. Then $\alpha_i$ must also satisfy $k_1\alpha_1 + k_3 \alpha_2 + k_2 \alpha_3=2\pi$. Since $\poly{T}$ must satisfy these two equations and $k_2 \neq k_3$, $T$ must be isosceles. Therefore $\poly{T}$ tiles in exactly the same way as $\poly{T}'$. Angles still sum to $2 \pi$ around every vertex, and the edges match since both triangles are isosceles. 
\end{proof}

\begin{proposition} \label{prop:generalized-gs}
A triangle $\poly{T}$ tiles with every angle at every vertex if its angles 
$\alpha_i$ satisfy $\sum k_i\alpha_i = 2\pi$ for $k_i \ge 2$ congruent modulo 
$2$.
\end{proposition}

\begin{proof}
Suppose $\poly{T}$ satisfies $\sum k_i\theta_i = 2\pi$ with $k_i \ge 2$ congruent modulo $2$. By \cref{lem:existence}, there exists a scalene triangle $\poly{T}'$ that satisfies $\sum k_i\theta_i = 2\pi$ for those $k_i$ and no other nonnegative integers. By \cref{thm:gs}, $\poly{T}'$ tiles with every angle at every vertex because each $k_i$ is positive, and by \cref{prop:unique-tiling}, $\poly{T}$ tiles in the same way.
\end{proof}

We can now use \cref{prop:generalized-gs} to obtain certain tilings in the hyperbolic plane. 

\begin{proposition}\label{prop:tri211}
A triangle $\poly{T}$ with angles $\theta_i$ such that 
$2k\theta_1 + \theta_2 +\theta_3 = \pi$ for some positive integer $k$ tiles $\H^2$.
\end{proposition}

\begin{proof}
This proposition follows immediately from \cref{prop:generalized-gs}.
\end{proof}

\begin{proposition} \label{triapi}
There is a non-equilateral isosceles triangular tile $\poly{T}$ of $\H^2$ for all possible triangular areas $A$, i.e., for $0<A<\pi$. 
\end{proposition}

\begin{proof}
Let $\poly{T}$ be the hyperbolic isosceles triangle with angles 
\begin{align*}
    \theta_1 &= \frac{ A}{2k - 1}, \\ 
    \theta_2=\theta_3 &= \pi/2 - k\theta_1,
\end{align*}    
    for some integer $k > \pi/(2\pi-2A)$ large enough to make $\theta_1<\theta_2=\theta_3$. 
By Gauss-Bonnet, $\poly{T}$ has area $A$. By Proposition \ref{prop:tri211}, $\poly{T}$ tiles. 
\end{proof}

\begin{corollary} \label{degentri}
    There is a $n$-gonal tile of $\H^2$ for any given area $0<A<\pi$.
\end{corollary}
\begin{proof}
    By \cref{triapi}, there exists a non-equilateral triangular tile $\poly{T}$ of area $A$. Choose a side of distinct length, and add $n-3$ equally spaced vertices to get a degenerate $n$-gonal tile of area $A$.
\end{proof}

\begin{remark}
\textcite[Thm.\ 5]{MM} explicitly construct strictly convex $n$-gonal tiles of every possible area $0<A<(n-2)\pi$ for $n \ge 5$. The tiling is generically nonperiodic, although invariant under a discrete group of symmetries. Their Theorem 4 constructs some equilateral tiles for all $n \geq 3$ by perturbing the regular $n$-gon and using \cref{prop:margulis} below.
\end{remark}

\begin{proposition} \label{rhombus}
There is a rhombic tile of $\H^2$ for all possible quadrilateral areas $A$, i.e. for $0 < A < 2 \pi$.
\end{proposition}

\begin{proof}
By \cref{triapi}, there exists a non-equilateral isosceles triangular tile of area $A/2$. Consider a tiling by this isosceles triangle. Pair tiles connected by the side of distinct length. Each pair of isosceles triangles forms the same rhombus of area $A$, and this rhombus tiles.
\end{proof}

\begin{remark}
    \textcite[Thm.\ 4]{MM} construct some rhombic tiles, but only for some areas, by perturbing the regular $4$-gon.
\end{remark}

\begin{conjecture} \label{quadrilateralgs}
A quadrilateral with distinct angles $\theta_i$ tiles if and only if  
\[ 
    \Sigma k_i\theta_i = 2\pi,
\]
for positive integers $k_i$ congruent modulo $2$, none or all of which are $1$. 
\end{conjecture}

The following proposition of \textcite{MM} gives a sufficient condition for equilateral $n$-gonal tiles, $n\geq 4$, which \citeauthor{MM} use to construct some aperiodic tiles.
Our \cref{prop:evenequilateral} provides a general construction of equilateral $n$-gons, and then our \cref{prop:n+2} constructs equilateral even-gonal tiles of a wide range of areas.

\begin{proposition}[\textcite {MM}, Prop.\ 2.2] \label{prop:margulis}
Let $\poly{Q}$ be a convex equilateral polygon in $\H^2$ with $n \geq 4$ vertices and angles $\theta_1,\dotsc,\theta_n$ at most $\pi/2$. Assume that any three angles (allowing repetition) may be complemented by more (allowing repetition) to sum to $2\pi$.
Then $\poly{Q}$ tiles $\H^2$.
\end{proposition}

To prove the existence of many $2n$-gonal \emph{tiles} in \cref{prop:n+2}, we need \cref{prop:evenequilateral} about the existence of equilateral $2n$-gons (\cref{fig:my_label}). Note that by Gauss-Bonnet, as the sum of half the angles approaches $(n-1)\pi$, the area goes to~$0$.
\begin{figure}[ht]
    \centering
    \includegraphics[width=0.5\textwidth]{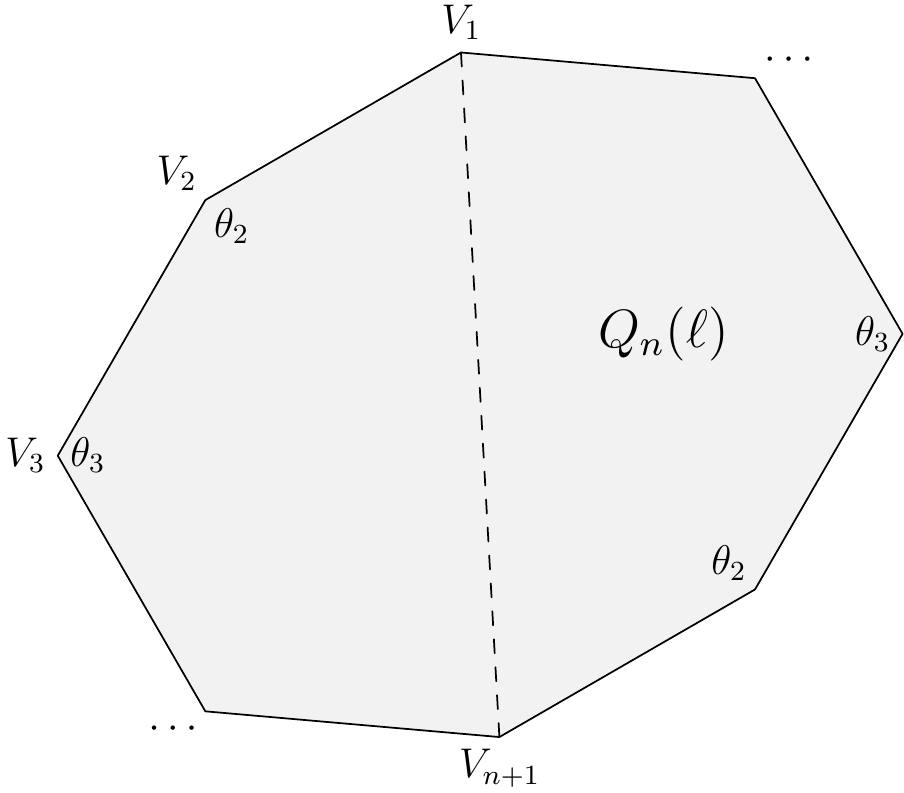}
    \caption{Construction of equilateral $2n$-gon with angles $\theta_1, \dotsc, \theta_n, \theta_1, \dotsc, \theta_n$.}
    \label{fig:my_label}
\end{figure}
\begin{proposition}\label{prop:evenequilateral}
Consider $\theta_1, \dots, \theta_{n}$ such that $0 < \theta_i \leq \pi$ and $\sum \theta_i < (n-1)\pi$. Then there is a convex equilateral $2n$-gon in $\H^2$ with angles $\theta_1, \dots, \theta_n, \theta_1, \dots, \theta_n$. 
\end{proposition}

First we need two lemmas.

\begin{lemma}\label{lemma:embedded}
Consider $\theta_2, \dots, \theta_n$ with $0 < \theta_i < \pi$. 
Let $\poly{Q}_n(\ell)$ denote the $(n+1)$-gon $V_1 \dots V_{n+1}$ of edge lengths $\ell$ with $m(V_i) = \theta_i$ for $i \in \{2, \dots, n\}$. Then 
\begin{enumerate}[label = {\textup{(\arabic*)}}]
\item for large $\ell$, $\poly{Q}_n(\ell)$ is embedded;
    
\item $m(V_1), m(V_{n + 1}) \to 0$ as $\ell \to \infty$;
    
\item $d(V_1, V_{n + 1}) \to \infty$ as $\ell \to \infty$. 
\end{enumerate}
\end{lemma}

\begin{proof}
For a proof by induction, first consider the case $n = 2$. For any $\ell$, the triangle $\poly{Q}_2(\ell)$ is trivially embedded since $\theta_2 < \pi$. The rest follows by induction and the hyperbolic Law of Cosines. 
For the induction step, for $\ell$ large, since by induction $\poly{Q}_{n - 1}(\ell)$ is embedded and $m(V_n)$ in $\poly{Q}_{n - 1}(\ell)$ is small, therefore $\poly{Q}_n(\ell)$ is embedded. Again the rest follows by the Law of Cosines.
\end{proof}

\begin{lemma}\label{lemma:gpi}
    Let $L$ be the supremum of $\ell$ such that $\poly{Q}_n(\ell)$ is not embedded. If $L > 0$, then for some $\ell_0 > L$, $m(V_1) + m(V_{n + 1}) > \pi$ in $\poly{Q}_n(\ell_0)$.
\end{lemma}

\begin{proof}
It follows from \cref{lemma:embedded} that for large enough $\ell > L$, $V_{n + 1}$ is outside $\poly{Q}_{n - 1}(\ell)$ since $m(\angle V_1 V_n V_{n - 1}) \to 0$ as $\ell \to \infty$, and symmetrically, $V_1$ is outside polygon $V_2 \dots V_{n + 1}$. $\poly{Q}_n(\ell)$ varies continuously with $\ell$, and as $\ell$ decreases, $\poly{Q}_n(\ell)$ is embedded as long as $V_{n + 1}$ is outside $\poly{Q}_{n - 1}(\ell)$ and $V_1$ is outside polygon $V_2 \dots V_{n + 1}$. Since $L > 0$, we may assume that for some $\ell_0 > L$, $V_{n + 1}$ is arbitrarily close to $\poly{Q}_{n - 1}(\ell)$, at which point 
\[\text{Area}(V_1 V_n V_{n + 1}) < \epsilon/2\]
and 
\[m(\angle V_1 V_n V_{n + 1}) < \epsilon /2,\]
with $0 < \epsilon < m(\angle V_2 V_1 V_n)$ on $\poly{Q}_{n}(\ell')$. Note that $\poly{Q}_n(\ell_0)$ is embedded since $\ell_0 > \ell \ge L$, and 
\[m(V_1) + m(V_{n + 1}) = \pi - \text{Area}(V_1 V_n V_{n + 1}) - m(\angle V_1 V_n V_{n + 1}) + m(\angle V_2 V_1 V_n) > \pi.\]
\end{proof}

\begin{proof}[Proof of \cref{prop:evenequilateral}]
Consider the polygonal chain $V_1 \cdots V_{n+1}$ where each edge is of length $\ell$ and each angle $V_i$ has measure $\theta_i$ for $2 \leq i \leq n$. 
By \cref{lemma:embedded}, the $(n+1)$-gon $\poly{Q}$ with vertices $V_1, \dotsc, V_{n+1}$ is embedded for sufficiently large $\ell$. 
Furthermore, $m(V_1) + m(V_{n+1})$ continuously approaches $0$ as large $\ell$ goes to infinity.

Let $L$ be the supremum of $\ell$ such that $\poly{Q}_n(\ell)$ is not embedded. Suppose $L > 0$. By \cref{lemma:gpi}, there exists an $\ell_0 > L$ such that $m(V_1) + m(V_{n + 1}) > \pi$. Since $m(V_1) + m(V_{n + 1}) \to 0$ as $\ell \to \infty$ and $\theta_1 < \pi$, there must exist an $\ell \in (\ell_0, \infty)$ such that $m(V_1) + m(V_{n + 1}) = \theta_1$ on $\poly{Q}_n(\ell)$. If $L = 0$, then $\poly{Q}_n(\ell)$ is embedded for every $\ell > 0$, so $m(V_1) + m(V_{n + 1})$ attains every value from $0$ to the Euclidean limit 
$(n - 1) \pi - \sum_{i = 2}^n \theta_i$. In either case, for some $\ell > 0$, $m(V_1) + m(V_{n + 1}) = \theta_1$ on $\poly{Q}_n(\ell)$. Therefore adjoining two copies of the chain $V_1\cdots V_{n+1}$ with length $\ell$ yields the desired $2n$-gon.
\end{proof}

\begin{proposition} \label{prop:n+2}
For even $n \geq 6$, there is a strictly convex equilateral $n$-gonal tile $\poly{Q}$ of\/ $\H^2$ of area $A$ for $ (n-2)\pi/2 < A < (n-2)\pi$.
\end{proposition}

\begin{proof}
Note first what turns out to be one exceptional case: the regular $6$-gon with $\pi/6$ angles tiles by \cref{prop:regularpoly}. It has area $3\pi$.

Let $\sigma = (n-2) - A/\pi$. By the hypothesis on $A$, $0 < \sigma < (n-2)/2.$ If $\sigma < 2(n-4)/(n-2)$, there is an integer $m$ such that
\begin{equation} \label{eq:m-bound}
\frac{4}{(n-2)\sigma}< m < \frac{2}{\sigma}
\end{equation}
because the length of the interval is greater than $1$. Otherwise let $m=4/(n-2)$. Note that
\[ m = \frac{4}{n-2} > \frac{4}{(n-2)\sigma} \] and
\[ m = \frac{4}{n-2} < \frac{2}{\sigma}, \]
so $m$ satisfies the same inequalities \cref{eq:m-bound} in this case. The sharp inequality in the lower bound holds because $\sigma \ge 2(n-4)/(n-2) \ge 1$, with equality only for the already handled case $n=6,A=3\pi.$

Let $\theta_1 = (\pi/m(n-4))(2 - m\sigma)$. Note that $0 < \theta_1 < 2\pi/m(n-2)$ by \cref{eq:m-bound} (and $\theta_1 < \pi/2).$ Finally, let $\theta$ be such that 
\begin{equation}
    (n - 2)(\theta_1+ \theta) = 2\pi /m. \label{eq:angle-sum}
\end{equation}
Note that $0<\theta<\pi/2$. By \cref{prop:evenequilateral}, there exists an equilateral $n$-gon $\poly{Q}$ with two angles of measure $\theta_1$ and the rest of measure $\theta$. Since the angles are all less than $\pi/2$, $\poly{Q}$ is strictly convex. By \cref{prop:Gauss-Bonnet}, \cref{eq:angle-sum}, the definition of $\theta_1$, and the definition of $\sigma$, 
\begin{align*}
\text{Area}(\poly{Q}) &= ( n -2)\pi - (2\theta_1 + (n-2)\theta) \\
& =(n-2)\pi - (2\pi/m-(n-4)\theta_1) \\
&=(n-2)\pi - \pi\sigma \\
&= A.
\end{align*}
If $m$ is integral, by \cref{eq:angle-sum} and \cref{prop:margulis}, $\poly{Q}$ tiles. In the case $m = 4/(n-2)$,
$$4(\theta_1+\theta) = 2\pi,$$
and again by \cref{prop:margulis}, $\poly{Q}$ tiles.
\end{proof}

\begin{remark}
For $n=6$, as the area approaches $3\pi$ from below, $\theta$ approaches 0 and $\theta_1$ approaches $\pi/2$, and as the area approaches $3\pi$ from above, $\theta_1$ approaches 0 and $\theta$ approaches $\pi/4$. Fortunately, this exceptional case is covered by a regular $6$-gon.
\end{remark}

\begin{corollary}\label{cor:degen}
For even $n \geq 4, k \geq 2$, there is a (degenerate) equilateral $kn$-gonal tile $\poly{Q}$ of\/ $\H^2$ of area $A$ for any $ (n-2)\pi/2 < A < (n-2)\pi.$ 

\end{corollary}
\begin{proof}
   Add $k-1$ equally-spaced vertices to each edge of the $n$-gonal tile guaranteed by \cref{rhombus} and \cref{prop:n+2}.
\end{proof}
\begin{corollary}\label{cor:2nto4n}
For any $n \equiv 2 \pmod 4$  at least $6$ and any $k \ge 1,$ there is a nondegenerate equilateral $kn$-gonal tile $\poly{Q}$ of\/ $\H^2$ of area $A$ for any $(n-2)\pi/2 < A < (n-2)\pi.$
\end{corollary}

\begin{proof}
Consider the tiling by the equilateral $n$-gon of \cref{prop:n+2} with angles \[\theta_1, \theta_2, \ldots, \theta_{n/2}, \theta_1, \theta_2, \ldots, \theta_{n/2}\] and desired area. The case $k=1$ is already done, so assume $k \ge 2.$
    Let $\poly{Q}$ be the $kn$-gon constructed by deforming the edges of the $n$-gon: add, in alternating fashion, an indent or an outdent to the edges of the $n$-gon, which evidently preserves area. The indents and outdents are congruent equilateral polygonal chains of $k$ edges, and can be made arbitrarily small to guarantee that $\poly{Q}$ does not intersect itself. We claim that $\poly{Q}$ tiles. Note that, as $n/2$ is odd, the edges between angles $\theta_i,\theta_{i+1}$ and $\theta_{i+n/2},\theta_{i+1+n/2}$ are dented differently: one has an outdent, and the other an indent.
    
    Consider the graph for the $n$-gonal tiling; we use it to generate an analogous tiling for $\poly{Q}$. There are no odd cycles, because a cycle bounds a collection of even-gons and the unused (interior) edges are paired up. Hence the graph is bipartite, consisting of two sets $C$ and $C'$.
    
   For a vertex of $C$, arrange all the dents clockwise about the vertex; for a vertex of $C'$, arrange them counter-clockwise. Since the dentings alternate, every face of this new graph is congruent to $\poly{Q}$.
\end{proof}

\begin{remark} 
Consider $12$-gons for example. \cref{prop:n+2} provides equilateral $12$-gonal tiles from the largest possible area  $10\pi$ down to $5\pi$ (excluding the endpoints). There are regular 12-gonal tiles for areas of the form $(10-24/k)\pi$ (\cref{prop:Gauss-Bonnet} and \cref{prop:regularpoly}), including for example $2\pi$ and $4\pi$. Adding triangular dents to the edges of equilateral $6$-gonal tiles as in \cref{cor:2nto4n} yields equilateral $12$-gonal tiles for areas from $4\pi$ down to $2\pi$. Evenly placing two vertices as in \cref{cor:degen} on each of the edges of a rhombic tile (\cref{rhombus}) yields (degenerate) tiles for areas from $2\pi$ down to $0$. The only missing cases are areas in the interval $(4\pi,5\pi]$. 
\emph{Non-equilateral} tiles are provided for all possible areas by \cref{degentri}. 
\end{remark}

\begin{proposition}\label{curvi-convex-tri}
    An isoperimetric curvilinear triangular tile of the hyperbolic plane must be convex.
\end{proposition}

\begin{proof}
Assume that there is a non-convex isoperimetric curvilinear triangular 
tile. If every edge contains the same area as a geodesic, replacing the
edges with geodesics maintains area and reduces perimeter, contradiction.
In the case that one contains more and another contains less, a similar contradiction is obtained. Hence either two edges contain more and one contains less, or two contain less and one contains
more. Then around a vertex of the tiling one type must match up against
the other type, so that all outside edges are of the same type, which 
leads to a contradiction around an outside vertex.
\end{proof}

\begin{remark}
\cref{curvi-convex-tri} is easier in closed hyperbolic surfaces, because the number of edges bulging out must equal the number bulging in, while in $\H^2$ such a discrepancy might be pushed off to infinity. Even in closed surfaces an extension to higher curvilinear $k$-gons remains conjectural, because straightening one edge of a tile might cause it to intersect another part of the tile. 
\end{remark}

\section{Monohedral Tilings of Closed Hyperbolic Surfaces} 
\label{sect:monohedral}
In 2005 \textcite{cox2005, cox2011} and subsequently \textcite{sesum} proposed generalizing Hales's hexagonal isoperimetric inequality to prove that a regular $k$-gons $\poly{R}_k$ ($k \geq 7$) with $120^\circ$ angles provides a least-perimeter tiling of an appropriate closed hyperbolic surface for given area. \textcite{carroll06} showed that the proposed polygonal isoperimetric 
inequality fails for $k > 66$. 
Our \cref{prop:heptagon-best} proves the result for $R_7$ among monohedral tilings by a polygon of at most $10$ sides.
Although \cref{prop:heptagon-best} applies even if the regular polygon does not tile, \cref{thm:existence-of-tilings} notes that there are many closed hyperbolic surfaces which it does tile. 
It is possible for many-sided polygons to tile, but \cref{cor:3} shows that
as $n$ increases, $n$-gonal tiles necessarily have many concave angles.
\cref{cor:heptagon-beats-convex} deduces that the regular polygon has less perimeter than any other \emph{convex} polygonal tile.

\begin{remark}
By Gauss-Bonnet, the regular $k$-gon $\poly{R}_k$ of area $A_k = (k-6)\pi/3$ ($k \geq 7$) has interior angles of $2\pi/3$ (\cref{sect:hyperbolic-geometry}). It therefore tiles $\H^2$. It also tiles many closed hyperbolic surfaces (\cref{thm:existence-of-tilings}). Every such $\poly{R}_k$ is thought to be isoperimetric. However, for area not a multiple of $\pi/3$, there is no conjectured isoperimetric tile.

\end{remark}

\begin{proposition}\label{thm:existence-of-tilings}
For $k\geq 7$, there exist infinitely many closed hyperbolic surfaces tiled by the regular $k$-gon of area $(k-6)\pi/3$ and angles $2\pi/3$.
\end{proposition}

\begin{proof}
These surfaces are provided by work of \textcite[Main Thm.]{edmonds-tess} on torsion-free subgroups of Fuchsian groups and tessellations (see also \cite{edmonds-bull, edmonds}). Their work yields torsion-free subgroups $S$ of arbitrarily large finite index of the triangle group $(2, 3, k).$ This triangle group is the orientation-preserving symmetry group of the hyperbolic triangle of angles $\pi/2, \pi/3,$ and $\pi/k$. Each quotient of $\H^2$ by such a subgroup $S$ is a closed hyperbolic surface tiled by these triangles, which can be joined in groups of $2k$ to form a tiling by the regular $k$-gon of area $(k-6)\pi/3$ and hence angles $2\pi/3$ (by Gauss-Bonnet).
	\end{proof}

\begin{example}
    The Klein Quartic Curve in $\C P^2$ is the set of complex solutions to the homogeneous equation \cite{klein-original}
    \[u^3v + v^3w + w^3u = 0.\]
    The curve is a hyperbolic 3-holed torus. It is famously tiled by 24 regular heptagons.
\end{example}
The following results are instrumental in eliminating competing $n$-gons of large~$n$. 
\begin{lemma}
\label{lemma:vertexdegrees}
Consider a tiling of a closed hyperbolic surface by curvilinear polygons $\poly{Q}_i$ of average area $A_k=(k-6)\pi/3$.
Then each polygon has on average at most $k$ vertices of degree at least $3$, with equality if and only if every vertex has degree two or three.
\end{lemma}

\begin{proof}
A tile with $n$ edges and $v$ vertices of degree at least $3$ contributes to the tiling $1$ face, $n/2$ edges, and at most $(n-v)/2 + v/3$ vertices, with equality precisely if no vertices have degree greater than $3$. Therefore its contribution to the Euler characteristic $F-E+V$ is at most $1-v/6$. The Gauss-Bonnet theorem says that 
\[ \int G=2\pi(F-E+V). \]
Hence the average contributions per tile satisfy
\[
    -A_k = -(k-6)\pi/3 \leq 2\pi(1-\overline{v}/6).
\]
Therefore $\overline{v} \leq k$, with equality if and only if no vertices have degree more than~$3$.
\end{proof}

\begin{proposition}
	\label{cor:3}
	Let $\poly{Q}$ be an $n$-gon of
	area $A_k=(k-6)\pi/3$
	with $\ell_1$ (interior) 
	angles of measure $\pi$ and $\ell_2$ of measure greater than $\pi$. 
	If  $\poly{Q}$ tiles $M$, then $\ell_1 + 2\ell_2 \ge n-k$. 
	Equality holds for a tiling (and therefore every tiling) if and only if every vertex is of degree two or three, and every concave angle has degree two.
\end{proposition}

\begin{proof}
    Take any tiling of $M$ by $\poly{Q}$. Each vertex of degree two in the tiling has either two angles of measure $\pi$ or exactly one angle of measure greater 
	than $\pi$.
    By Lemma~\ref{lemma:vertexdegrees},
        $$\ell_1 + 2\ell_2 \ge n-k,$$
    with equality precisely when every vertex has degree two or three, and every 
	concave angle has degree 2.
\end{proof}

\begin{corollary}\label{cor:heptagon-beats-convex}
    The regular $k$-gon $\poly{R}_k$ has less perimeter than any non-equivalent convex polygonal tile of area $A_k = (k-6)\pi/3$.
\end{corollary}

\begin{proof}
    Let $\poly{Q}$ be a convex $n$-gonal tile of area $A_k$. By \cref{cor:3}, $\poly{Q}$ contains at least $n-k$ angles of measure $\pi$. Hence $\poly{Q}$ is equivalent to a polygon with at most $k$ sides. Unless $\poly{Q}$ is equivalent to $\poly{R}_k$, $\poly{Q}$ has strictly more perimeter by \cref{prop:3-k}.
\end{proof}

\section{Octagonal Tiles} \label{sec:octagonal-tiles}

Our original approach to proving the regular 7-gon $\poly{R}_7$ with $120^\circ$ angles isoperimetric, having eliminated $n$-gonal competitors for $n \leq 7$, next took up $8$-gons. \cref{cor:hept-beats-octa} proves that the regular heptagon of area $\pi/3$ has less perimeter than any octagonal tile of the same area. Since strictly convex tiles of this area
do not exist for $n \geq 8$, we turn our attention to octagonal
tiles which are not strictly convex.

\begin{proposition}\label{prop:octagons}
    The regular heptagon of area $\pi/3$ has less perimeter than any non-equivalent
    non-strictly-convex octagon of the same area.
\end{proposition}

\begin{proof}
    The proposition is immediate from \cref{cor:n-7concave}.
\end{proof}

\begin{corollary}\label{cor:hept-beats-octa}
    Let $M$ be a closed hyperbolic surface that is tiled by the regular 
    heptagon $\poly{R}_7$ of area $\pi/3$. Then $\poly{R}_7$ has less
    perimeter than any non-equivalent octagonal tile of the same area.
\end{corollary}

\begin{proof}
    By \cref{cor:3}, the octagon contains an angle of measure at least $\pi$. The corollary follows from \cref{prop:octagons}.
\end{proof}

\section{Nonagonal Tiles} \label{sec:nonagonal-tiles}

Proving that the regular heptagon has less perimeter than any $9$-gonal tile
(\cref{prop:7-beats-9}) is more difficult than the octagonal case because we must fully consider what 
happens when the tile
has strictly concave angles. 
\cref{cor:flattening-comp-vertices} first proves that ``flattening''
degree-two concave angles and their corresponding convex angles reduces 
perimeter while preserving area.

\begin{definition}[Flattening]
Consider a polygonal chain $A_1A_2 \dots A_n$ in $\textbf{H}^2$. To flatten adjacent vertices $A_2 \dots A_{n-1}$, replace $A_1A_2, \dots, A_n$ with the geodesic $A_1A_n$. In a hyperbolic surface, flattening is done in the cover $\H^2$. Let $m(A)$ for a vertex of a polygon denote the measure of the interior angle of vertex $A$.
\end{definition}

\begin{lemma}\label{lem:complementary-vertices}
    Let $M$ be a surface which admits a monohedral tiling by a polygon $\poly{Q}$.
    Suppose that $\poly{Q}$ has a degree-$2$ vertex $v$ with measure $m(v)$. 
    Then $\poly{Q}$ also has a vertex $w$ of measure $2\pi - m(v)$.
    If $\poly{Q}$ has no angles of measure $\pi$, then $v$ and $w$ are distinct vertices.
    Furthermore, the incident edges of $v$ are equal in length to the incident edges of $w$.
\end{lemma}

\begin{proof}
    In the tiling, vertex $v$ on $\poly{Q}$ has measure $m(v)$, and since it is degree-2 it
    is shared by exactly one other copy of $\poly{Q}$ in the tiling; on this other 
    tile, $v$ has measure $2\pi - m(v)$. Since the tiling is monohedral, all 
    tiles are congruent, and so there must exist a vertex of measure 
    $2\pi - m(v)$ on $\poly{Q}$ has well.
    If $\poly{Q}$ contains no angles of measure $\pi$, then it is not possible
    that $m(v) = 2\pi - m(v)$, and so $v$ and $w$ must be distinct vertices.
    Since tilings are edge-to-edge, it must be the case that the edges
    incident to $v$ coincide with edges incident to $w$, and so they are 
    equal in length.
\end{proof}

\begin{corollary} \label{cor:cong-triangles}
    Let $B,B'$ be distinct complementary vertices on a monohedral tile $\poly{Q}$.
    Let $A,C$ be the vertices adjacent to $B$ and let $A', C'$ be those adjacent to $B'$.
    Then $ABC$ is congruent to $A'B'C'$.
\end{corollary}

\begin{proof}
    This follows immediately from \cref{lem:complementary-vertices}.
\end{proof}

\begin{corollary} \label{cor:cong-subtriangles}
    Let $B,B'$ be distinct but adjacent complementary vertices 
    on a monohedral tile $\poly{Q}$. Let $A$ be the other vertex adjacent to 
    $B$ and $C$ be the other vertex adjacent to $B'$.
    Let $D$ be the intersection of the segments $BB'$ and $AC$.
    Then $\triangle ABD$ is congruent to $\triangle DB'C$.
\end{corollary}

\begin{proof}
    This follows immediately from \cref{lem:complementary-vertices}.
\end{proof}

\begin{corollary} \label{cor:flattening-comp-vertices}
    Flattening distinct complementary vertices $B,B'$ of a tile $\poly{Q}$ 
    does not change the area of $\poly{Q}$.
\end{corollary}

\begin{proof}
    Without a loss of generality, let $m(B) < \pi$.
    If $B$ and $B'$ are adjacent, then flattening them amounts to removing the
    area of $\triangle ABD$ and adding the area of $\triangle DB'C$ to $\poly{Q}$,
    as shown in \cref{cor:cong-subtriangles}.
    Since these triangles are congruent, the area of $\poly{Q}$ does not change.
    If $B$ and $B'$ are not adjacent, then flattening them amounts to removing
    the area of $\triangle ABC$ and adding the area of $\triangle A'B'C'$ 
    to $\poly{Q}$, as shown in \cref{cor:cong-triangles}. 
    Since these triangles are congruent, the area of $\poly{Q}$ does not change.
\end{proof}

\begin{proposition} \label{prop:7-beats-9}
    Let $M$ be a closed hyperbolic surface.
    Then the regular heptagon $\poly{R}_7$ of area $\pi/3$ has less perimeter than any non-equivalent $9$-gonal tile $\poly{Q}$ of $M$ of the same area.
\end{proposition}

\begin{proof}
    Suppose $\poly{Q}$ has an angle of measure $\pi$. If there is only one such angle then by
    \cref{cor:3}, $\poly{Q}$ is equivalent to an octagon with at least one strictly
    concave angle. 
    By \cref{prop:octagons}, 
    $P(\poly{Q}) > P(\poly{R}_7)$.
    If there are two or more angles of measure $\pi$, $\poly{Q}$
    is equivalent to a polygon with seven or fewer sides, and so by
    \cref{prop:3-k}, $P(\poly{Q}) > P(\poly{R}_7)$.
    
    On the other hand, suppose that $\poly{Q}$ does not have an angle of measure $\pi$.
    By \cref{lemma:vertexdegrees}, there exist distinct vertices $B, B'$ on $\poly{Q}$
    such that $m(B) + m(B') = 2\pi$ and the edges incident to each vertex are equal in
    length. Let $A$ and $C$ be the vertices adjacent to $B$ and let $A'$ and $C'$ be
    those adjacent to $B'$. Since $\poly{Q}$ has no angles of measure $\pi$,
    let $m(B) < \pi$ without loss of generality. 
    Then there are no vertices in
    the interior of $\triangle ABC$ and $B$ is the only vertex in the interior of 
    $\triangle A'B'C'$, since otherwise the convex hull $H(\poly{Q})$ would be 
    a polygon with seven or fewer sides, so 
    $P(\poly{Q}) > P(H(\poly{Q})) \geq P(\poly{R}_7)$, and we are done.
    Let $\poly{Q}'$ be the heptagon formed by flattening both $B$ and $B'$, 
    whether or not they are adjacent.
    By \cref{cor:flattening-comp-vertices}, the area of $\poly{Q}$ is equal to
    that of $\poly{Q}'$, and the perimeter is reduced.
    Since $\poly{Q}'$ is a heptagon, it has at least as much perimeter as $\poly{R}_7$, so 
    $P(\poly{Q}) > P(\poly{Q}') \geq P(\poly{R}_7)$.
\end{proof}

\section{Decagonal Tiles} \label{sec:decagonal-tiles}

Proving that the regular heptagon has less area than any $10$-gonal tile (\cref{prop:heptagon-beats-decagon}) is yet more difficult than the $9$-gonal case since there may be multiple concave
angles, which means we need to also worry about the adjacency of the angles
on the tile. This in turn requires case work to address every possible
configuration of angles on the $10$-gon.

\begin{definition}
	Let $B$ be a vertex of a polygon. Then $\triangle B$ is the triangle formed
	by connecting $B$ and its adjacent vertices.
\end{definition}

\begin{lemma}[Angle Nesting]
	Let $\poly{Q}$ be a decagon with $2$ concave angles $A$, $A'$ with 
	corresponding convex angles $V, V'$. Then $P(\poly{Q}) > P(\poly{R}_7)$
	if either
	\begin{enumerate}[label={\textup{(\arabic*)}}]
		\item there exists a vertex in the interior of $\triangle A$; or
		\item there exists a vertex in the interior $\triangle V$ or 
		      $\triangle V'$ which is neither $A$ nor $A'$.
	\end{enumerate}
\end{lemma}

\begin{proof}
	We show that in either case, the convex hull $\hull (\poly{Q})$ contains
	at least three vertices in its interior, and so has at most seven sides.
	\begin{enumerate}
		\item First, consider when $A'$ is inside $\triangle A$. Notice that 
		      the line connecting $AA'$ intersects $\poly{Q}$ at some point 
			  $M$, which is neither $A$ nor $A'$, that is inside $\triangle A$. 
			  If $M$ is a vertex of $\poly{Q}$ then $M$, $A$, and $A'$ are in 
			  the interior of $\hull (\poly{Q})$; otherwise, one of the two 
			  vertices $V$ of the edge on which $M$ lies must be in the 
			  interior of $\triangle A$, and so $V$, $A$, and $A'$ are in the 
			  interior of $\hull (\poly{Q})$.
			  
			  Suppose now that any vertex $V \not= A'$ is inside $\triangle A$.
			  Then $V$, $A$, and $A'$ all are in the interior of 
			  $\hull (\poly{Q})$.
		\item The only vertices which may lie in the interior of $\triangle V$ or $\triangle V'$ are $A$ and $A'$. If not, the convex hull must then contain the vertex inside $V$ or $V'$ as well as $A$ and $A'$, contradiction since there are three vertices in the interior of the convex hull. 
	\end{enumerate}
	Since $\poly{Q}$ is not equivalent to $\poly{R}_7$, $\hull (\poly{Q})$ is 
	a polygon with at most seven sides. If $\hull(\poly{Q}) = \poly{R}_7$,
	then since $\poly{Q}$ is not equivalent to $\poly{R}_7$, $\poly{Q}$ must
	contain a strictly concave angle, and so 
	$P(\poly{Q}) > P(\hull (\poly{Q})) = P(\poly{R}_7)$.
	On the other hand, if $\hull(\poly{Q})$ is an irregular heptagon or is
	a polygon with fewer than seven sides then we know that 
	$P(\poly{Q}) \geq P(\hull (\poly{Q})) > P(\poly{R}_7)$.
\end{proof}

\begin{remark}
	This shows that we may flatten the pairs of concave and convex angles
	of such a decagon without self-intersections.
\end{remark}

\begin{proposition} \label{prop:heptagon-beats-decagon}
	Let $M$ be a closed hyperbolic surface. Then the regular heptagon 
	$\poly{R}_7$ of area $\pi/3$ has less perimeter than any non-equivalent 
	$10$-gonal tile $\poly{Q}$ of $M$ of the same area.
\end{proposition}

\begin{proof}
    By \cref{cor:n-7concave}, it suffices to consider only decagons with two or
	fewer concave angles. These angles cannot both have measure $\pi$, since
	otherwise by \cref{prop:Gauss-Bonnet} (Gauss-Bonnet), the remaining eight
	angles would have an average measure of $17\pi/24 > 2\pi/3$, implying they
	could not meet in threes and $\poly{Q}$ would have to contain another 
	concave angle.
	
	Suppose $\poly{Q}$ contains a single angle $A$ of measure $\pi$. Then there exists an angle $A'$ of measure greater than $\pi$ and a vertex $V'$ that fits into $A'$. Further, by
	\cref{lemma:vertexdegrees}, $\poly{Q}$ has an average of at least $3$ degree-two vertices per tile, so $A$, $A'$ and $V'$ always meet in twos. 
	First suppose $A$ and $V'$ are not adjacent. Let $\poly{Q}'$ be the polygon formed by flattening $A'$ and $V'$, and then taking the convex hull of the resulting shape. Then $\poly{Q}'$ is a heptagon of area $A\geq\pi/3$ and less perimeter. 
	If instead $A$ and $V'$ are adjacent, since $A$ must always meet in twos and therefore at $A$ on another copy of $\poly{Q}$, $A'$ is adjacent to $A$ as well. Flattening all three of $V', A, A'$ forms a heptagon $\poly{Q}'$ of area $A=\pi/3$ and less perimeter, as in \cref{fig:aapv}.
	In either case, by \cref{prop:reg-is-best},
	$P(\poly{Q}) < P(\poly{Q}') \leq P(\hull (\poly{Q}')) \leq P(\poly{R}_7)$.
	
    \begin{figure}[h]
        \centering
        \includegraphics{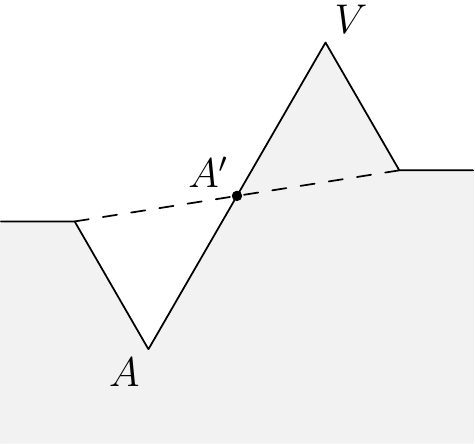}
        \caption{Flattening $AA'V$ reduces perimeter while preserving area.}
        \label{fig:aapv}
    \end{figure}
	
	Finally, we consider the case when $\poly{Q}$ contains two strictly concave
	angles $A,A'$, and two distinct corresponding strictly convex angles $V,V'$.
	If $A'$ is not adjacent to $A$ or $V$, then by flattening $A'$ would turn $\poly{Q}$ into a
	$9$-gon $\poly{Q'}$ with concave angle $A$ and corresponding angle $V$ that fits into $A$. By \cref{prop:7-beats-9}, $\poly{Q}'$ has more 
	perimeter than $\poly{R}_7$, so
	$P(\poly{Q}) > P(\poly{Q}') \geq P(\poly{R}_7)$. By symmetry, this also
	covers the case that $A$ is not adjacent to $A'$ or $V'$. If, however, $A'$ 
	\emph{is} adjacent to $A$ or $V$ (or, by symmetry, $A$ is adjacent to $A'$ 
	or $V'$), we enumerate the six possible orientations of the vertices and 
	show that the claim holds for each. Cases~\ref{case:1} and \ref{case:2} cover when $A'$ is adjacent to $A$ but not $V$; cases~\ref{case:3} and \ref{case:4} cover when $A'$ is adjacent to $V$ but not $A$; and cases~\ref{case:5} and \ref{case:6} cover when $A'$ is adjacent to both $A$ and $V$. 
	In the following proof, ``$\xymatrix@1@C=7pt{\ar@{-}[r] & }$'' is used to denote vertices that are not $V, A, V'$, or $A'$.
	
	\begin{enumerate}[label=(\arabic*)]
		\item $\xymatrix@1@C=7pt{\ar@{-}[r] & A'\!A & \ar@{-}[l]}$: \label{case:1}
		Flatten $A'\!A$ as in \cref{fig:aap} and flatten $V$. This reduces perimeter and increases area, because the 
		triangle removed at $V$ is congruent to the triangle added at $A$. 

\begin{figure}[ht]
\centering
\includegraphics{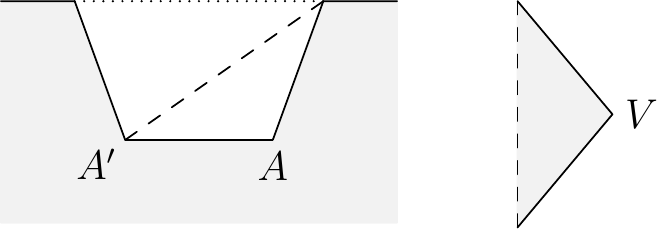}
\caption{In the \ptops{$\xymatrix@1@C=7pt{\ar@{-}[r] & A'\!A & \ar@{-}[l]}$}{--A'A--}
case, flattening reduces perimeter and increases area, because the triangle removed at $V$ is 
congruent to a triangular portion of the trapezoidal region added at $A'\!A$.}
\label{fig:aap}
\end{figure}

\item $\xymatrix@1@C=7pt{A'\!AV}$: \label{case:2} Flatten $AV$ as in the dashed line of \cref{fig:apav}, preserving area and reducing perimeter. Note that $A'$ remains concave after flattening $AV$. Taking the convex hull (dotted line) yields a polygon with seven or fewer sides, and so $P(\poly{Q}) > P(H(\poly{Q}')) \geq P(\poly{R}_7)$
Note that $V'$ may occur anywhere without affecting the argument, including adjacent to $A'$.

\begin{figure}[ht]
\centering
\includegraphics{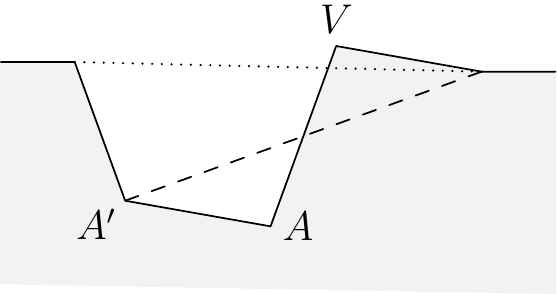}
\caption{In the $A'AV$ case, the dashed-line flattening is followed by taking the convex hull (dotted line).}
\label{fig:apav}
\end{figure}

\item $\xymatrix@1@C=7pt{V'\!AA'V}$: \label{case:3}
By \cref{lemma:vertexdegrees}, the average number of degree-two vertices per tile is at least 3, which means that some copy of $\poly{Q}$ must have both $A$ and $A'$ degree two. When $V$ fits into $A$, $A'$ cannot fit into itself (recall $A'>\pi$), so $A'$ must simultaneously fit into another $V'$. Thus necessarily the other angle adjacent to $V$ is congruent to $V'$, and so this
reduces to the following Case~\ref{case:4}, as illustrated in \cref{fig:vpaapvvp}.
	
\begin{figure}[h]
\centering
\includegraphics{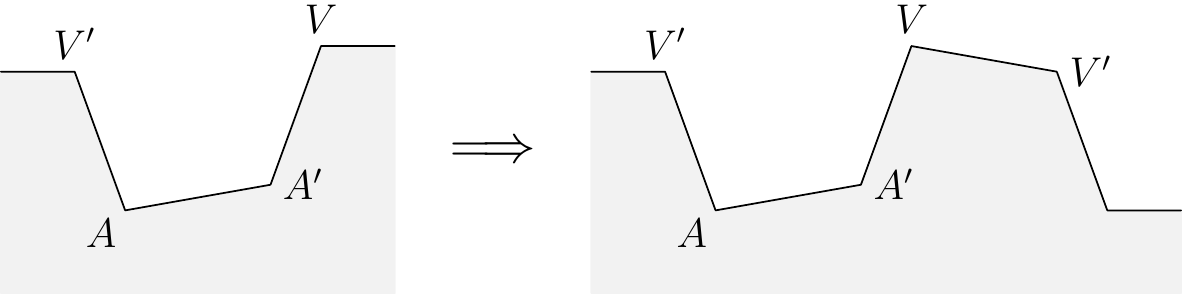}
\caption{The $V'\! A A' V$ configuration necessarily implies that there is an angle congruent to $V'$ adjacent to $V$, reducing to Case~\ref{case:4}.}
\label{fig:vpaapvvp}
\end{figure}

\item $\xymatrix@1@C=7pt{AA'VV'}$ and $\xymatrix@1@C=7pt{AA'V \ar@{-}[r] & V'}$: \label{case:4}
Flatten $A'$ and $V'$ as in \cref{fig:aapvvp}, preserving area and reducing perimeter. Note that the angle at $A$ remains concave. Take the convex hull of the resulting polygon and the result is heptagon or less with at least as much area and less perimeter than $Q$. Therefore $P(Q) >  P(Q_7)$. 

\begin{figure}[ht]
\centering
\includegraphics{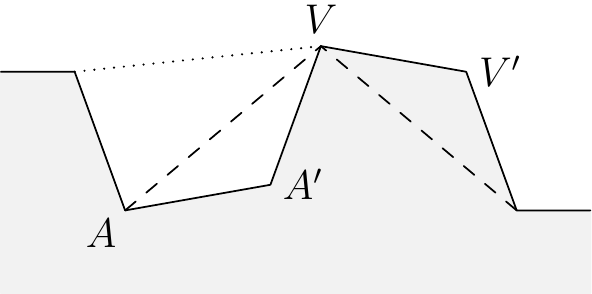}
\caption{In the $AA'VV'$ case, the dashed-line flattening is followed by taking the convex hull (dotted line).}
\label{fig:aapvvp}
\end{figure}

\item $\xymatrix@1@C=7pt{V\!A' \ar@{-}[r] & AV'}$ \label{case:5}
Note that every concave angle must be part of a degree 2 vertex, so the polygonal curve consisting of the three edges incident to $V$ and $A'$ must be congruent to the polygonal curve with edges incident to $A$ and $V'$. Therefore flattening $VA'$ and $AV'$ yields a heptagon or less with the same area and less perimeter. Therefore $P(\poly{Q}) > P(\poly{R}_7)$. 

\begin{figure}[ht]
\centering
\includegraphics{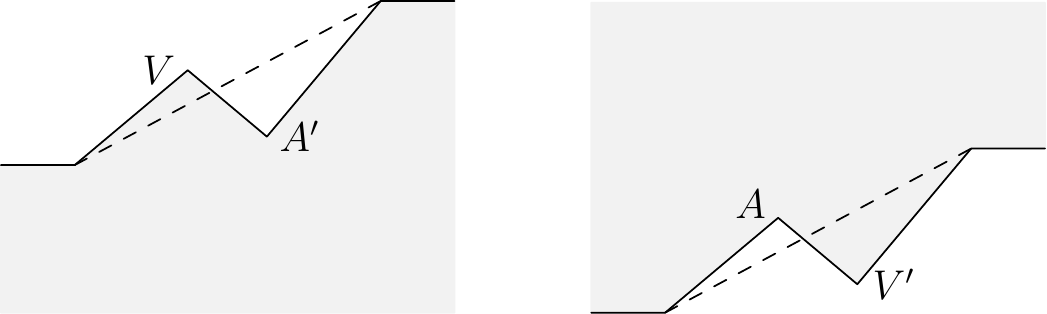}
\caption{The polygonal chain containing $V$ and $A'$ is congruent to the one containing $A$ and $V'$, so we may flatten simultaneously (dashed line) with no net change in area.}
\label{fig:vpaavp}
\end{figure}

\item $\xymatrix@1@C=7pt{A'V\!AV'}$ \label{case:6}
Flatten $VA$ as in \cref{fig:apvavp}, which preserves area and reduces perimeter. The measure of angle $A'$ will decrease, but the measure of angle $V'$ will increase by the same amount. This guarantees that $A'$ or $V'$ will have measure at least $\pi$. Taking the convex hull yields a heptagon or less with at least as much area and less perimeter than $\poly{Q}$. Thus
$P(\poly{Q}) > P(\poly{R}_7)$.

\begin{figure}[h]
\centering
\includegraphics{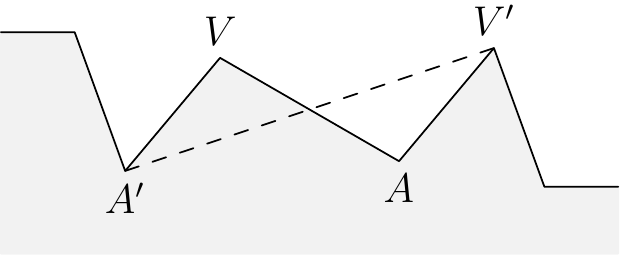}
\caption{In the $A'V\!AV'$ case, the dashed-line flattening is followed by taking the convex hull.}
\label{fig:apvavp}
\end{figure}

\end{enumerate}
	
Since these six cases enumerate all possible permutations of the relevant angles, the proof is complete.
\end{proof}

The following theorem is our main result.
\begin{theorem} \label{prop:heptagon-best}
	Let $M$ be a closed hyperbolic surface. Then the regular heptagon 
	$\poly{R}_7$ of area $\pi/3$ has less perimeter than any non-equivalent 
	$n$-gonal tile of $M$ of the same area for $n \leq 10$.
\end{theorem}
\begin{proof}
The theorem follows from \cref{prop:3-k,cor:hept-beats-octa,prop:7-beats-9,prop:heptagon-beats-decagon}
\end{proof}

\section{11-gonal Tiles}
Finally our work stalls with partial results on $11$-gonal tiles. They are similar to decagonal tiles in that they have at least two concave angles. However, they can have more concave angles, and there are more possible permutations of the concave $A_i$ and the corresponding convex $V_i$. We first employed casework based on the number of $A_i$, and further subdivided based on the number of angles exactly equal to $\pi$. The most difficult cases are when there are three concave angles, one of which might never meet in twos, with multiple copies of some $V_i$, and several subcases of exactly two concave angles. We resolved all but three of around 20 cases for $11$-gonal tiles before discovering a general proof. We leave the reader with several examples of the $11$-gon casework. 

\begin{lemma}
An $11$-gonal tile $\poly{Q}$ with exactly two concave angles $A_1, A_2$ must have $A_1, A_2$ meet in twos for every copy of $\poly{Q}$. 
\end{lemma}

\begin{proof}
By \cref{lemma:vertexdegrees}, there must be on average at least $11-7=4$ degree-two vertices per tile. The result follows. 
\end{proof}

\begin{corollary}
\label{11gontwopi}
An $11$-gonal tile $\poly{Q}$ with exactly one strictly concave angle $A_1$ and two angles $m(A_2)=m(A_3)= \pi$ must have $A_1$, $A_2$, $A_3$ meeting in twos for every copy of $\poly{Q}$.
\end{corollary}

\begin{proof}
The result follows from \cref{lemma:vertexdegrees}. 
\end{proof}

\begin{corollary}[Angle Measures]
\label{cor:angle-measures}
Note that an $11$-gonal tile $\poly{Q}$ cannot have exactly one strictly concave angle and exactly one angle of measure $\pi$, as there could not be an average of $4$ degree-two vertices per tile. Similarly, $\poly{Q}$ cannot have exactly three concave angles all of which have measure $\pi$.
\end{corollary}

\begin{lemma}[Four or More Concave Angles]
A non-equivalent $11$-gonal tile $\poly{Q}$ of area $\pi/3$ with four or more concave angles is worse than $\poly{R}_7$.
\end{lemma}

\begin{proof}
    Suppose $\poly{Q}$ has four or more concave angles. By \cref{cor:n-7concave}, $P(\poly{Q})\geq P(\poly{R}_7)$, with equality if and only if $\poly{Q}\sim\poly{R}_7$.
\end{proof}

\begin{proposition}[Three Concave Angles]
A non-equivalent $11$-gonal tile $\poly{Q}$ of area $\pi/3$ with exactly three concave angles is worse than $\poly{R}_7.$
\end{proposition}

\begin{proof}
Consider such a $\poly{Q}$ with concave angles $A_1, A_2, A_3$ and corresponding convex angles $V_1, V_2, V_3$. 
We first consider the cases where some of the $A_i$ have measure $\pi$.
By \cref{cor:angle-measures}, we may assume without loss of generality $m(A_1)>\pi$. 
\begin{enumerate}
    \item If $m(A_2)=m(A_3)=\pi$, by \cref{11gontwopi}, each $A_i$ always meets in twos. Removing $A_2$ and $A_3$ forms a $9$-gonal \emph{tile} $\poly{Q}'$ since both vertices always met in twos and thus were only ever aligned with each other. Since $\poly{Q}'$ is a $9$-gonal tile, it reduces to Section $7$.
    \item If only $m(A_3)=\pi$, consider its neighbors. If none of $A_i,V_i, i\leq 2$ neighbor $A_3$, remove $A_3$, resulting in a $10$-gon which, while not a tile, satisfies the concave and convex angle requirements of a $10$-gon, and thus, with a little more work, can be shown to be sub-optimal.  \qedhere
\end{enumerate}
\end{proof}

The following represents some incomplete results necessary for the $11$-gon proof.
\begin{lemma} 
$A_i$ neighbors $A_3$ if and only if $V_i$ neighbors $A_3$.
\end{lemma}
\begin{proof}
Without loss of generality, assume $A_1$ neighbors $A_3$. Assume $V_1$ does not. Since on average $\poly{Q}$ has four degree-two vertices per tile, $A_1$ must sometimes meet in twos. When $A_3$ meets in twos—and therefore meets $A_3$ on another copy of $\poly{Q}$—both copies of $A_1$ cannot meet in twos, since the only other neighbor of $A_3$ is not $V_1$. But then $A_3$ adds two degree-two vertices to the overall sum, but subtracts both potential $A_1$ and $V_1$, a total of four. Thus the average is too small, and $V_1$ must neighbor $A_3$.

Now without loss of generality assume $V_1$ neighbors $A_3$. Similar to the above, when $A_3$ meets in twos and meets $A_3$ on another copy of $\poly{Q}$, both copies of $V_1$ cannot meet in twos, which leaves two other copies of $A_1$ unfilled as well. A separate case covers when there is more than $1$ copy of $V_1$; if so, that's advantageous, as we can then use the other $V_1$ instead. Again, this makes the average too small, so $A_1$ also neighbors $A_3$. Flattening $V_1, A_1$ and $A_3$ as in the diagram forms an $8$-gon with at least one convex angle ($A_2$). Taking the convex hull to form $\poly{Q}'$ yields a $7$-gon with equal or greater area. Hence $P(\poly{R}_7) < P(\poly{Q}') < P(\poly{Q})$. 
\end{proof}
More casework would be necessary to fully resolve the case of $11$-gons. 

\section{Euclidean Hexagons}
A subsequent paper \cite{hirsch} simultaneously proves \cref{con:main-result} in comparison with polygons of any number $n$ of sides, generalizes the result from $7$ to all $k \geq7$, and remarks that the same methods yield a relatively simple proof of a weak version (\cref{hexagons}) of Hales's theorem \cite{hales} on Euclidean hexagons. Here we provide the details behind the extension to Euclidean hexagons. The following propositions and lemmas \ref{euc-vertexdegrees}--\ref{euc-A(n/2)} provide (generally easier) Euclidean versions of the hyperbolic cases presented in \cite[Lemma~4.3, Proposition~5.3, Lemma~5.4, and Lemma~5.5]{hirsch}.
\begin{lemma}
\label{euc-vertexdegrees}
Consider a tiling of a flat torus by curvilinear polygons. Then each polygon has on average at most $6$ vertices of degree at least $3$, with equality if and only if every vertex has degree two or three.
\end{lemma}

\begin{proof}
A tile with $n$ edges and $v$ vertices of degree at least $3$ contributes to the tiling $1$ face, $n/2$ edges, and at most $(n-v)/2 + v/3$ vertices, with equality precisely if no vertices have degree greater than $3$. Therefore it adds at most $1-v/6$ to the Euler characteristic $F-E+V$. The Gauss-Bonnet theorem says that 
\[ 
    \int G=2\pi(F-E+V). 
\]
Hence the average contributions per tile satisfy
\[
0 \leq 2\pi(1-\overline{v}/6).
\]
Therefore $\overline{v} \leq 6$, with equality if and only if no vertices have degree more than~$3$.
\end{proof}
\begin{proposition}\label{euc-cover}
Let $M$ be a flat torus tiled by curvilinear polygons $Q_i$. Let $\poly{Q}_i^*$ be the convex hull of the vertices of degree three or higher of $Q_i$. Then $\{\poly{Q}_i^*\}$ covers $M$ and the average number of sides is less than or equal to $6$. 
\end{proposition}
\begin{proof}
By the Euclidean restatement of \cite[Lemma 5.2]{hirsch}, straightening edges and flattening all degree-$2$ vertices yields a covering by immersed polygons, each covered by the corresponding $Q_i^*$. Hence $\{\poly{Q}_i^*\}$ covers $M$. By \cref{euc-vertexdegrees}
the average number of sides is less than or equal to $6$.
\end{proof}
\begin{proposition}
\label{lem:euc-area-n}
The area of the regular $n$-gon with perimeter $P$ is given by \[A(n)=\frac{P^2\cot\alpha}{4n},\] where $\alpha=\pi/n.$ The function $A(n)$ is strictly increasing and strictly concave on $[2,\infty).$ We extend $A(n)$ continuously to be identically $0$ on the interval $[0,2].$
\end{proposition}
\begin{proof}
Let $R$ be the circumradius of the regular $n$-gon of perimeter $P$. Its area is \[ \frac{n}{2}R^2\sin2\alpha.\] But \[\sin\alpha=\frac{P}{2Rn},\] and a simple substitution yields the claimed expression for $A(n)$. Its second derivative with respect to $n$ is \[\frac{P^2}{4} \cdot \frac{n^2[2\cot\alpha\left(1+\alpha^2\csc^2\alpha)\right)-4\alpha \csc^2\alpha]}{n^5}.\]

The numerator can be rewritten as
\[\frac{P^2n^2}{\sin^3\alpha} \cdot \left(2\cos \alpha\cdot (\sin^2\alpha+\alpha^2)-4\alpha\sin\alpha\right) \tag{\text{$\star$}}.\]

The derivative (with respect to $\alpha$) of the term in parentheses
\[-2\alpha^2 \sin\alpha-6\sin^3\alpha,\] is negative over $0<\alpha\le \pi/2.$ Since the term in the parentheses is zero at $\alpha=0,$ it follows that $(\star)$ and hence the second derivative of $A(n)$ are negative for $0 < \alpha \le \pi/2.$

Finally, strict monotonicity of $A(n)$ follows from strict concavity, since $A(n)$ remains positive for $n>2.$
\end{proof}
\begin{lemma}
\label{euc-A(n/2)}
Fix $P>0.$ For all real $n\ge 6,$ \[A(n) <2A\left(\frac{n}{2}\right).\]
\end{lemma}
\begin{proof}
The desired inequality simplifies to
\[\cot(\pi/n) < 4\cot(2\pi/n),\]
and for $n>4$ further rearranges to
\[\frac{2}{3}<\cos^2(\pi/n),\]
which is true for $n \ge 6$.
\end{proof}

\begin{proposition}
\label{hexagons}
Consider a curvilinear polygonal tiling of a flat torus with $N$ tiles of average area $A$ and no more perimeter than the regular hexagon $R_6$ of area $A$. Then every tile is equivalent to $R_6.$
\end{proposition}

\begin{proof}
Let $P$ be the perimeter of the regular hexagon of area $A$. By \cref{euc-cover}, the collection of convex hulls $\poly{Q}_i^*$ of the vertices with degree at least $3$ on each tile covers $M$, and of course $P(\poly{Q}_i^*) \leq P(\poly{Q}_i) \leq P$ by assumption.  
Since the $\poly{Q}_i^*$ cover, 
\begin{equation}
\label{euc-ave-area}
\frac{1}{N}\sum \text{Area}(\poly{Q}_i^*) \ge A.
\end{equation}
By \cref{euc-cover}, the number of sides $n_i$ of $Q_i^*$ satisfy 
\[\frac{1}{N}\sum n_i \le 6.\]
The areas can be estimated in terms of $A(n)$ for $P$ as
\begin{equation}
\sum \text{Area}(\poly{Q}_i^*) \le \sum A(n_i) \le
N \cdot A \left(\frac{\sum n_i}{N} \right) \le 
N \cdot A(6) = 
N \cdot A.
\label{eq:euc-avehullarea}
\end{equation}
The first inequality follows from the well-known fact that regular (Euclidean) $n$-gons maximize area for given perimeter.
The second inequality follows from the concavity of $A(n)$ for $n\geq 2$ (\cref{lem:euc-area-n}) and Jensen's inequality. 
If any of the $n_i$ are $0$ or $1$, choose some $n_i \ge 6$, and use \cref{euc-A(n/2)} first to replace $0+A(n_i)$ with $2A(n_i/2)$. 
If you run out of large enough $n_i$, the next inequality holds already.
The third inequality follows from the fact that $A(n)$ is strictly increasing (again \cref{lem:euc-area-n}). The final equality holds by the definition of $A(n)$ for $P$.

By \cref{euc-ave-area}, equality must hold in every inequality.
By the strict concavity of $A(n)$, equality in the second inequality implies that every $n_i = 6$. Equality in the first inequality implies that every $\poly{Q}_i^*$ has area $A$. Since regular hexagons uniquely maximize area, $\poly{Q}_i^*$ is the regular hexagon $\poly{R}_6$ of area $A$. Finally 
\[P(\poly{Q}_i) \ge P(\poly{Q}_i^*) = P,\]
and equality implies that $\poly{Q}_i \sim \poly{R}_6$.
\end{proof}

\printbibliography

\end{document}